\documentclass{article}

\usepackage{PRIMEarxiv}

\usepackage[utf8]{inputenc} 
\usepackage[T1]{fontenc}    
\usepackage{hyperref}       
\usepackage{url}            
\usepackage{booktabs}       
\usepackage{amsfonts}       
\usepackage{nicefrac}       
\usepackage{microtype}      
\usepackage{lipsum}
\usepackage{fancyhdr}       
\usepackage{graphicx}       
\graphicspath{{media/}}     

\usepackage[normalem]{ulem} 

\usepackage{graphics} 
\usepackage{epsfig} 
\usepackage{times} 
\usepackage{amsmath} 
\usepackage{amssymb}  
\let\labelindent\relax
\usepackage{enumitem}   
\usepackage{hyperref}
\usepackage{algorithm}
\usepackage{bbold}
\usepackage{xcolor}
\usepackage{mathtools}
\usepackage{algpseudocode}
\usepackage{amsthm}

\newtheorem{definition}{Definition}[section]

\newtheorem{lemma}{Lemma}
\newtheorem{assumption}{Assumption}
\newtheorem{remark}{Remark}
\newtheorem{problem}{Problem}

\newtheorem{proposition}{Proposition}
  
\title{{Nonlinear Wasserstein Distributionally Robust Optimal Control}
}

\author{
  Zhengang Zhong\\
    Imperial College London\\
  London, UK\\
  z.zhong20@imperial.ac.uk \\
   \And
  Jia-Jie Zhu \\
  Weierstrass Institute for Applied Analysis and
Stochastics \\
  Berlin, Germany\\
  zhu@wias-berlin.de \\
}

\begin{document}
\maketitle

\begin{abstract}
This paper presents a novel approach to addressing the distributionally robust nonlinear model predictive control (DRNMPC) problem. Current literature primarily focuses on the static Wasserstein distributionally robust optimal control problem with a prespecified ambiguity set of uncertain system states. Although a few studies have tackled the dynamic setting, a practical algorithm remains elusive. To bridge this gap, we introduce an DRNMPC scheme that dynamically controls the propagation of ambiguity, based on the constrained iterative linear quadratic regulator.
The theoretical results are also provided to characterize the stochastic error reachable sets under ambiguity.
We evaluate the effectiveness of our proposed iterative DRMPC algorithm by comparing the closed-loop performance of feedback and open-loop on a mass-spring system. 
Finally, we demonstrate in numerical experiments that our algorithm controls the propagated Wasserstein ambiguity.
\end{abstract}

\section{Introduction}
\subsection{Background and motivation}
Model predictive control (MPC) repeatedly solves optimization problems online based on a system model and prescribed constraints to determine optimal control actions \cite{mayne2000constrained}. However, the closed-loop performance of MPC designed based on the nominal system model could be severely deteriorated when the real system suffers from uncertainty
\cite{cannon2009model}.

To effectively develop control methods addressing the detrimental effect of uncertainty, two classes of MPC that explicitly take the uncertainty
into account have emerged: stochastic MPC (SMPC) and robust MPC (RMPC). RMPC determines the optimal control actions under
the worst-case scenario
within a pre-specified deterministic uncertainty set \cite{mayne2005robust}, whereas SMPC assumes or estimates the distribution of the uncertainty and selects the best control action for an 
objective function under soft constraints \cite{mayne2014model}. However, the performance of RMPC might be over-conservative as low-probability uncertainty is also taken into account, whereas the actual performance of SMPC could significantly deviate from the designed one due to the distribution discrepancy between the true distribution and the nominal distribution used in the controller design \cite{heirung2018stochastic}.

For the purpose of addressing the challenges mentioned above - conservativeness or misspecified nominal distribution - we consider a data-driven distributionally robust nonlinear MPC (DRNMPC) problem using the Wasserstein metric. In the construction of this controller, instead of knowing the 
probability distribution of disturbances exactly, only samples of the disturbance realizations are required
to construct the Wasserstein ambiguity set. 
The ambiguity set includes the empirical distribution of disturbance samples at its center and all distributions within a certain Wasserstein distance. Control actions are determined based on the worst-case distribution from this set, considering distributional robustness.

\begin{figure}[thpb]
  \centering
\includegraphics[width=0.4\textwidth, height = 0.2\textheight]{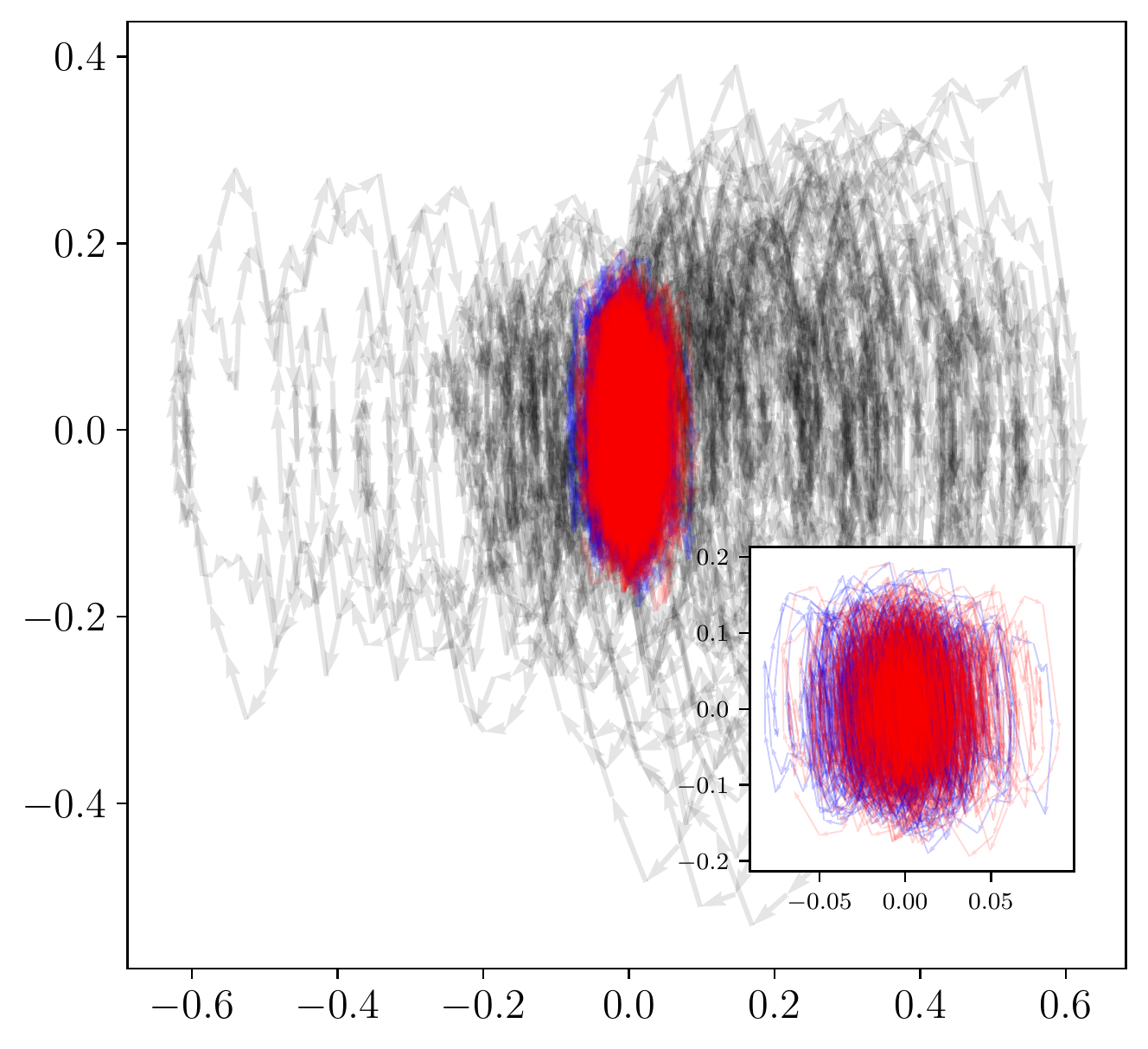}
    \caption{
    Closed-loop error dynamics (as in tube MPC) of 30 realizations with the feedback gains and nominal inputs solved by \ref{alg:1}. Red: Our method. Blue: Fixed feedback gain. Black: Zero feedback gain, i.e. open-loop. The arrow indicates the error vector between two consecutive sampling times, i.e. the tail indicates the accumulated error of all the previous steps and the head indicates the accumulated error including the error from the current step. 
    See Section~\ref{sec:err-reachable} for the theoretical characterization.} \label{fig:iterative_LQR_error_prop}
\end{figure}

\subsection{Related work}
Recently, distributionally robust control using the Wasserstein ambiguity garners a lot of interest and attention. For the purpose of state constraint satisfaction, the recent papers \cite{mark2020stochastic,  zhong2021data, coulson2021distributionally, micheli2022data, fochesato2022data} consider such a distributionally robust MPC problem with respect to the Wasserstein ambiguity set defined on the product probability space for linear systems, wherein the center of the ambiguity set is determined based i.i.d. samples of disturbance sequences. Both \cite{hakobyan2022wasserstein} and \cite{yang2020wasserstein} consider the distributionally robust control problem as a two-player zero-sum game without state constraints and solve the problem via dynamic programming with a relaxed formulation using a Wasserstein penalty. For nonlinear systems, \cite{zolanvari2022data} considers deterministic systems with disturbed constraints, and \cite{zhong2023tube} solves distributionally robust MPC for nonlinear systems with additive disturbances via iterative linearization. However, in \cite{zhong2023tube}, the propagation of the ambiguity set is not considered and the feedback gain is static for each sampling time.
After the initial submission of this manuscript, we were brought to the awareness of a recent preprint \cite{aolaritei2023capture}, which is the closest to our work. 
While both works consider the dynamic setting in terms of Wasserstein ambiguity, we directly formulate the DRO problem based on the disturbance ambiguity dynamically, which is mathematically equivalent to their propagation to the state distributions while enjoying simpler forms; cf. \eqref{eq-error-decomposition}, \eqref{eq:reform_DR_constraints}.
Furthermore, compared with \cite{aolaritei2023capture}, this paper solves DRNMPC based on iterative LQR, whereas they considered linear systems.

In this work, we consider Wasserstein distributionally robust MPC for nonlinear systems with additive disturbances. Instead of constructing the Wasserstein ambiguity set for disturbance sequences, we consider the Wasserstein ambiguity set of disturbance for the single-step dynamics and propagate the Wasserstein ambiguity sets within the prediction horizon. Also, instead of using a relaxed formulation, we solve the original DRNMPC problem via an iterative method with the help of Riccati recursion \cite{rawlings2017model}. We will show that our method could dynamically control the propagation of the Wasserstein ambiguity sets and hence guarantee a non-conservative closed-loop performance.

\subsection{Contribution}
This paper makes the following main contributions. 
1) We solve a Wasserstein distributionally robust nonlinear model predictive control (DRNMPC) problem for nonlinear systems  \eqref{eq:iLQR_prototype_prob}. 
To the best of our knowledge, this is the only work that does not assume a priori Wasserstein ambiguity sets of the state distributions for nonlinear systems.
2) We present an iterative-linearized DRMPC scheme that uses feedback to dynamically control the propagation of Wasserstein ambiguity sets, whereas open-loop control fails to do so. The derivation of such an approach is summarized in Proposition \ref{prop:1} and the corresponding algorithm is introduced in Algorithm \ref{alg:1}. This approach is a significant improvement over the existing literature, as previous research only addressed static problems with fixed ambiguity sets or dynamic problems without a practical algorithm. To the authors' best knowledge, our proposed algorithm is the first to provide a practical and efficient method for controlling the propagation of Wasserstein ambiguity sets in nonlinear dynamics.
3) We analytically characterize the Wasserstein distributional reachable set under dynamic propagation in our algorithm in Proposition \ref{prop-error-reachable-set}.
4) We visualize the closed-loop performance of the proposed approach via an error diagram in fig \ref{fig:iterative_LQR_error_prop}. We observe that our method effectively controls the propagation of the ambiguity sets.

The rest of the paper is organized as follows. In Section II, we introduce the control problem and the preliminary DRNMPC. Section III describes the Wasserstein ambiguity set applied to this work. In Section IV, we introduce the propagation of Wasserstein ambiguity sets and the corresponding algorithm dynamically controls the propagation. Also, we analyze the reachable sets of dynamic Wasserstein ambiguity and linearization error. In Section V, we provide a numerical experiment of a mass-spring system to demonstrate our method and comparison results.

\section{Problem statement}
\subsection{Notations}
We use $x_{k}$ for the measured state at time $k$ and $x_{i \mid k}$ for the state predicted $i$ steps ahead at time $k$. $[A]_{j}$ and $[a]_{j}$ denote the $j$-th row and entry of the matrix $A$ and vector $a$, respectively. Similarly, we denote the element of i-th row and j-th column in the matrix $A$ as $[A]_{ij}$. We define the notation $[A]_{i:j}$ for the i-th to j-th row in the matrix $A$. The set $\mathbb{N}_{>0}$ denotes the positive integers and $\mathbb{N}_{\geq 0}=\{0\} \cup \mathbb{N}_{>0}$. The set $\mathbb{N}_{1}^{N}$ denotes the set of integers from $1$ to $N$. $\mathcal{M}(\Xi)$ defines the space of all probability distributions supported on $\Xi$ with finite first moments. $(\cdot)^{(i)}$ denotes the i-th sample from the training set. The sequence of length $N$ of vectors $v_{0 \mid k}, \ldots, v_{N-1 \mid k}$ is denoted by $\mathbf{v}_{N \mid k}$. $\gamma_{ij}$ denotes the element of a 2-D tensor, such that this element is the i-th, j-th element along the first and second axis, respectively. Similar for 1-D $\lambda_i$. 
Let $\mathbb{B}_{\infty}^{n_x}:=\left\{d \in \mathbb{R}^{n_x} \mid\|d\|_{\infty} \leq 1\right\}$  denote the unit ball. Let $\mathbb{P}^{\otimes i} : = \mathbb{P}_0 \times \dots \times \mathbb{P}_{i-1}$ denote the product distribution.

\subsection{System dynamics, constraints and objective}
We consider the nonlinear time-invariant dynamical system with additive disturbance
\begin{equation}
\label{eq:system}
x_{k+1}=f_d(x_{k}, u_{k}) + w_{k}, \quad k \in \mathbb{N}_{\ge 0},
\end{equation}
where $f_d: \mathbb{R}^{n_x} \times \mathbb{R}^{n_u} \rightarrow \mathbb{R}^{n_x}$ is a discrete-time nonlinear dynamics, $k$ is the discrete sampling time, the state $x_{k} \in \mathbb{R}^{n_x}$, the control $u_{k} \in \mathbb{R}^{n_u}$, and the additive disturbance $w_{k} \in \mathbb{R }^{n_x}$. Each disturbance $w_k$ of the disturbance sequence $\{ w_{k} \}_{k \in \mathbb{N}_{\ge 0}}$ is assumed to be a realization of the corresponding random variable (r.v.) $W_{k}$ from the random process $\{W_{k}\}_{k \in \mathbb{N}_{\ge 0}}$ satisfying the following assumption.
\begin{assumption}[Bounded i.i.d Random Disturbance]
\label{assump:iid}
All random variables $W_k\sim\mathbb{P}_{w}$ for $k \in \mathbb{N}_{\ge 0}$ from the family of random variables $\{W_{k}\}_{k \in \mathbb{N}_{\ge 0}}$ are assumed to be zero-mean and independent and identically distributed (i.i.d) with an unknown probability distribution $\mathbb{P}_{w}$ and a known polyhedral support $\mathbb{W}_{w} \triangleq \{w \mid H_w w \le h_w\}$.
\end{assumption}


For any given state measurement $x_k$ at the sample time $k$, the predicted system states within the prediction horizon $N$ are described as 
$$
x_{i+1 \mid k}=f_d(x_{i \mid k}, u_{i \mid k})+ W_{i\mid k}, \quad x_{0 \mid k} = x_{k},
$$
where $x_{i \mid k}$, $u_{i \mid k}$, and $W_{i \mid k} := W_{k+i}$ are all random variables.

We further introduce the nonlinear dynamics $f_{d,i}: \mathbb{R}^{n_x} \times \underbrace{\mathbb{R}^{n_u} \times \dots \times \mathbb{R}^{n_u}}_{\text{i times}} \times \underbrace{\mathbb{R}^{n_x} \times \dots \times \mathbb{R}^{n_x}}_{\text{i times}} $ for the predicted state $x_{i \mid k}$ with $i \ge 1$ dependent on the measurement $x_k$, input sequence $u_{0 \mid k}, \dots, u_{i-1 \mid k}$, and disturbance sequence $W_{0 \mid k}, \dots, W_{i-1 \mid k}$
\begin{equation}
    \begin{aligned}
    x_{i \mid k} &= f_{d,i}(x_k, u_{0 \mid k}, \dots, u_{i-1 \mid k}, W_{0 \mid k}, \dots, W_{i-1 \mid k})\\
    & := f_d(f_d(\dots f_d(x_k, u_{0 \mid k})+ W_{0\mid k}  \cdots), u_{i-1 \mid k})+ W_{i-1 \mid k}.
    \end{aligned}
\end{equation}
To highlight that the predicted state is dependent on the disturbance sequence, we use a slight abuse of notation and denote 
$$
f_{d,i}(x_k, u_{0 \mid k}, \dots, u_{i-1 \mid k}, W_{0 \mid k}, \dots, W_{i-1 \mid k}) := x_{i \mid k}(W_{0 \mid k}, \dots, W_{i-1 \mid k}).
$$

For any nonlinear system, we consider distributionally robust constraints with ambiguity set propagation to the states 
\begin{equation}
\label{eq:proto_constraint}
\begin{aligned}
\sup_{\mathbb{P}_{m} \in \mathcal{P}_{k + m}, m = 0, \dots, i-1} \mathbb{E}_{\mathbb{P}^{\otimes i}}\left\{ [F]_{n}x_{i \mid k}(W_{0 \mid k}, \dots, W_{i-1 \mid k})\right\}  \leq [f]_{n}, \\
\end{aligned}
\end{equation}    
where $W_{m \mid k}\sim\mathbb{P}_{m}$ is the disturbance variable, $\mathbb{P}^{\otimes i} : = \mathbb{P}_0 \times \dots \times \mathbb{P}_{i-1}$, $k \in \mathbb{N}_{\ge 0}, n \in \mathbb{N}_1^{n_F}, i \in \mathbb{N}_1^N, F \in \mathbb{R}^{n_F \times n_x}, f \in \mathbb{R}^{n_F}$.
For each additive disturbance $W_{i\mid k}$ within the prediction horizon, we centered an ambiguity set $\mathcal{P}_{k + i}$ as the Wasserstein ball around the empirical distribution $\hat{\mathbb{P}}_{k + i}:=\frac{1}{M} \sum_{l=1}^{M} \delta_{\hat{w}_{i \mid k}^{(l)}}$. Due to the i.i.d assumption, the realization of additive disturbance is time-independent; hence, we will denote the ambiguity set as $\mathcal{P}$ and the corresponding empirical distribution as $\hat{\mathbb{P}}$ without explicitly indicating the predicted step $i$.
\begin{remark}
Through our formulation of the worst-case distributionally robust state constraints, the predicted states are affected by the accumulated error of the worst-cast distributions from each previous step within the prediction horizon. Hence, the control actions will be determined with an explicit consideration of the propagated effect of the worst distribution at each step of prediction. More details will be introduced in Section \ref{sec:prop_amb_ilqr}.
\end{remark}

Without loss of generality, we characterize the control target as tracking the equilibrium point, which we assume to be the origin of the coordinate system, from an initial state while satisfying the prespecified constraints. The control objective is hence defined as the minimization of the expected value with the reference trajectory uniformly equal to zero
\begin{equation}
\label{eq:proto_obj}
\mathbb{E}_{\mathbb{P}}\left\{\sum_{i=0}^{N-1}(\left\|x_{i \mid k}\right\|_{Q}^{2}+\left\|u_{i \mid k}\right\|_{R}^{2}) + \left\|x_{N \mid k}\right\|_{Q_f}^{2}\right\}.
\end{equation} Here $Q, Q_f \in \mathbb{R}^{n_x \times n_x}$ and $\mathbb{R}^{n_u \times n_u} $ are penalty matrices for the quadratic stage costs. The corresponding optimization problem of DRNMPC for nonlinear systems is defined as 

\begin{problem}
\label{eq:prototype_prob}
\begin{equation}
\label{eq:iLQR_prototype_prob}
\begin{array}{cl}
\displaystyle\min_{\mathbf{u}} & \mathbb{E}_{\mathbb{P}}\left\{\sum_{i=0}^{N-1}(\left\|x_{i \mid k}\right\|_{Q}^{2}+\left\|u_{i \mid k}\right\|_{R}^{2}) +\left\| x_{N \mid k}\right\|_{Q_f}^{2}\right\} \\
\text { s.t. } 
& x_{0\mid k}=x_{k}\\
& x_{i+1 \mid k}=f_d(x_{i \mid k},u_{i \mid k}) + W_{k+i}\\
    &\begin{aligned}
        \sup_{
        \mathbb{P}_{m} \in \mathcal{P}_{k + m}, m = 0, \dots, i-1
        } \mathbb{E}_{
        \mathbb{P}^{\otimes i}}\{&[F]_{n}x_{i \mid k}(W_{0 \mid k}, \dots, W_{i-1 \mid k})\} \leq [f]_{n}, \\
    \end{aligned}\\
& \forall \quad i \in \mathbb{N}_1^N, n \in \mathbb{N}_1^{n_F},  k \in \mathbb{N}_{\ge 0}.\\ 
\end{array}
\end{equation}
where $W_{m \mid k}\sim\mathbb{P}_{m}$ is the disturbance variable.
\end{problem}

\section{Distributionally robust optimization and Wasserstein Ambiguity Sets}
\label{sec:DRO_ambiguity}
Distributionally robust optimization is an optimization framework that utilizes partial information about the underlying probability distribution of the random variables in a stochastic model.
We consider the Wasserstein ambiguity set \cite{zhao2018data, mohajerin2018data} in this paper, which is modelled as a Wasserstein ball centered at a discrete empirical distribution. 
The Wasserstein metric defines the distance between all probability distributions $\mathbb{Q}$ supported on the uncertainty set $\mathbb{W}_{\xi} \in \mathbb{R}^{n_{\xi}}$ with finite $p$-moment, i.e.  $\int_{\mathbb{W}_{\xi}}\|\xi\|^{p} \mathbb{Q}(d \xi)<\infty$. 

\begin{definition}[Wasserstein Metric \cite{ambrosio2005gradient}]
\label{def:Wassertein_metric}
The \emph{Wasserstein metric} of order $p \ge 1$ is defined as $d_w: \mathcal{M}(\mathbb{W}_{\xi}) \times \mathcal{M}(\mathbb{W}_{\xi}) \rightarrow \mathbb{R}$ for all distribution $\mathbb{Q}_1, \mathbb{Q}_2 \in \mathcal{M}(\mathbb{W}_{\xi})$ and arbitrary norm on $\mathbb{R}^{n_\xi}$:
\begin{equation}
d^{p}_{w}\left(\mathbb{Q}_{1}, \mathbb{Q}_{2}\right):=
\inf_{\Pi} \int_{\mathbb{W}_{\xi}^{2}}\left\|\xi_{1}-\xi_{2}\right\|^{p} \Pi\left(\mathrm{d} \xi_{1}, \mathrm{~d} \xi_{2}\right),
\label{eq:wasserstein_metric}
\end{equation}
where $\Pi$ is a joint distribution of $\xi_{1}$ and $\xi_{2}$ with marginals $\mathbb{Q}_{1}$ and $\mathbb{Q}_{2}$ respectively.
\end{definition}

The Wasserstein metric originates from the optimal transportation problem \cite{villani2009optimal}, which studies the most efficient way to transport the mass of a distribution to another. In \eqref{eq:wasserstein_metric}, the Wasserstein distance between the distribution $\mathbb{Q}_1$ and $\mathbb{Q}_2$ can be interpreted as the minimal cost spent on the allocation if the Euclidean norm is selected and $p=2$. In the following, we will regard one distribution as the empirical distribution and the other as one of the unknown distributions which we assess whether to include or not in the ambiguity set. All these unknown distributions, whose distance from the empirical distribution is lower than a certain value in the Wasserstein sense, are included in the ambiguity set.

Specifically, we will only consider the type-1 Wasserstein metric in the remainder of this paper, i.e. $p= 1$.
In principle, it is also possible to use other $p$ values given the corresponding reformulation techniques \cite{kuhn2019wasserstein}. 
Then we could define the ambiguity set $\mathcal{P}$  centered at the empirical distribution leveraging the Wasserstein metric as
\begin{equation}
\mathbb{B}_{\varepsilon}\left(\hat{\mathbb{P}}\right):=\left\{\mathbb{Q} \in \mathcal{M}(\mathbb{W}_{\xi}): d_{w}\left(\hat{\mathbb{P}}, \mathbb{Q}\right) \leq \varepsilon\right\}
\label{eq:Wasserstein_ball}
\end{equation}
which specifies the Wasserstein ball with radius $\varepsilon>0$ around the discrete empirical probability distribution $\hat{\mathbb{P}}$. $\mathcal{M}(\mathbb{W}_{\xi})$ denotes the set of Borel probability measures on $\mathbb{W}_{\xi}$. The empirical probability distribution  $\hat{\mathbb{P}}:=\frac{1}{M} \sum_{l=1}^{M} \delta_{\hat{\xi}^{(l)}}$ is the mean of $M$ Dirac distributions which concentrates mass at the disturbance realization $\hat{\xi}^{(l)} \in \mathbb{W}_{\xi}$. We denote the training set of offline collected realizations $\xi$ as $\hat{\Xi}_{M}:=\left\{\hat{\xi}^{(l)}\right\}_{l \in \mathbb{N}_1^{M}} \subset \mathbb{W}_{\xi}$, which contains $M$ observed disturbance realizations.

The radius $\varepsilon$ determines the size of the Wasserstein ball \eqref{eq:Wasserstein_ball}, of which the size has been argued from various statistical perspective in the literature \cite{zhao2018data,rahimian2022frameworks, blanchet2019robust}
. Furthermore, as a function of the radius, the solution of this Wasserstein ambiguity-based DRO lies between the classical robust optimization and sample average approximation \cite{mohajerin2018data}.

\section{Propagation of ambiguity sets for nonlinear systems: Iterative distributionally robust LQR}
\label{sec:prop_amb_ilqr}
In this section, we propose an algorithm to solve the optimal control problem \eqref{eq:iLQR_prototype_prob} leveraging the techniques of Wasserstein ambiguity set propagation with dynamic feedback gains and iterative Linear Quadratic Regulator (iLQR) \cite{li2004iterative}.

For any given sampling time $k$, we solve the optimal control problem \eqref{eq:iLQR_prototype_prob} via sequentially deriving the linearized system and the corresponding feedback gains via Riccati recursion \cite[Sec. 8.8.3]{rawlings2017model}, solving a perturbed nominal problem, and updating the perturbations based on the solution from a DRO problem. Such an iterative method could be similarly found in, for example, \cite{messerer2021efficient}. However, our method significantly differs from the robust MPC approach \cite{messerer2021efficient} mainly in two aspects: (1) Instead of propagating the state uncertainty based on the ellipsoid support set of additive disturbances, we propagate the Wasserstein ambiguity through the dynamics. (2) We consider soft constraint satisfaction in expectation instead of robust constraint satisfaction.

\subsection{LTV formulation and error dynamics}
We first consider predicting the system dynamics with the help of tube-based MPC and a linear time-varying (LTV) error system as used in tube-based RMPC \cite{leeman2023robust}. 

We consider the predicted nonlinear dynamics \eqref{eq:system} in the form of the first-order Taylor series expansion:
\begin{equation}
\label{eq:linearized_with_error_apped2}
\begin{aligned}
x_{i + 1\mid k} & =f_d(x_{i \mid k}, u_{i \mid k}) + W_{i \mid k}\\
& = f_d(z_{i \mid k},v_{i \mid k}) + A(z_{i \mid k}, v_{i \mid k})(x_{i \mid k}-z_{i \mid k})  + B(z_{i \mid k}, v_{i \mid k})(u_{i \mid k}-v_{i \mid k})  + r(x_{i \mid k}, u_{i \mid k}, z_{i \mid k}, v_{i \mid k})+ W_{i \mid k},
\end{aligned}
\end{equation}
where $A(z_{i \mid k}, v_{i \mid k}):=\left.\frac{\partial f_d}{\partial x}\right|_{(x, u)=(z_{i \mid k}, v_{i \mid k})}$, $B(z_{i \mid k}, v_{i \mid k}):=\left.\frac{\partial f_d}{\partial u}\right|_{(x, u)=(z_{i \mid k}, v_{i \mid k})}$
and the remainder $r: \mathbb{R}^{n_{\mathrm{x}}} \times \mathbb{R}^{n_{\mathrm{u}}} \times \mathbb{R}^{n_{\mathrm{x}}} \times \mathbb{R}^{n_{\mathrm{u}}} \mapsto \mathbb{R}^{n_{\mathrm{x}}}$. 

Let $\Delta x_{i \mid k}:=x_{i \mid k}-z_{i \mid k}, \Delta u_{i \mid k}:= u_{i \mid k} - v_{i \mid k}$ denote the errors between nominal and real quantities, we have the following LTV error system
\begin{equation}
\label{eq:LTV_error_sys}
\Delta x_{i \mid k}=A_{i \mid k} \Delta x_{i \mid k}+B_{i \mid k} \Delta u_{i \mid k}+W_{i \mid k} + r_{i \mid k}, \Delta x_0=0_{n_{\mathrm{x}}},
\end{equation}
with $A_{i \mid k}:=A\left(z_{i \mid k}, v_{i \mid k}\right), B_{i \mid k}:=B\left(z_{i \mid k}, v_{i \mid k}\right), r_{i \mid k} := r(x_{i \mid k}, u,_{i \mid k} z_{i \mid k}, v_{i \mid k})$.
Furthermore, we apply the following control policy with dynamic feedback gain $K_{i \mid k}$ at each step
\begin{equation}
\label{eq:control_law}
\begin{aligned}
u_{i \mid k}=K_{i \mid k} x_{i \mid k}+c_{i \mid k},\\
\end{aligned}
\end{equation}
where $c_{i \mid k} \in \mathbb{R}^{n_u}$ are  decision variables in the optimal control problem.
Also, we have the corresponding nominal policy
\begin{equation}
\label{eq:control_law_nominal}
\begin{aligned}
v_{i \mid k}=K_{i \mid k} z_{i \mid k}+c_{i \mid k},\\
\end{aligned}
\end{equation}
Given the error dynamics \eqref{eq:LTV_error_sys}, control policy \eqref{eq:control_law} and nominal policy \eqref{eq:control_law_nominal}, we have
\begin{equation}
\begin{aligned}
\Delta x_{i+2 \mid k} &= A_{cl, i+1}\Delta x_{i+1 \mid k} + W_{i+1 \mid k} + r_{i+1 \mid k}\\
& = A_{cl, i+1}(A_{cl, i}\Delta x_{i \mid k} + W_{i \mid k} + r_{i \mid k} ) + w_{i +1 \mid k} + r_{i + 1 \mid k},
\end{aligned}
\end{equation}
where $A_{cl,i} := A_{i \mid k} + B_{i \mid k}K_{i \mid k}$.
Next, let 
$$
\begin{array}{ll}
e_{i+1 \mid k}=A_{cl, i} e_{i \mid k}+ W_{i \mid k} & e_{0 \mid k} = 0. \\
\varepsilon_{i + 1 \mid k}=A_{cl, i } \varepsilon_{i \mid k} + r_{i \mid k}& \varepsilon_{0 \mid k} = 0. \\ 
\end{array}
$$
If $x_{0 \mid k} = z_{0 \mid k}$, i.e. $\Delta x_{0 \mid k} = 0$, we have 
\begin{equation}
\label{eq:error_dynamic_append2}
    \Delta x_{i \mid k} = e_{i \mid k} + \varepsilon_{i \mid k}, \forall i \in \mathbb{N}_{0}^{N}
\end{equation}
by induction.

\begin{remark}
    In this paper, we consider three cases of feedback gains: fixed feedback with zero gain (open-loop control), fixed feedback with stabilizing gain, and feedback gain computed using iLQR. We will show that if the feedback gain is zero (i.e. open-loop control), the closed-loop performance is significantly worse than with fixed feedback gain and dynamic feedback gain as the size of the propagated ambiguity sets cannot be effectively controlled under open-loop control.
\end{remark}

\subsection{Distributionally robust nonlinear model predictive control}

Given the error dynamics \eqref{eq:error_dynamic_append2}, we consider solving an approximated version of the prototype DRNMPC problem \eqref{eq:iLQR_prototype_prob}.
In this subsection, similar to \cite{messerer2021efficient}, we consider the approximated dynamics till the first-order approximation - i.e. ignore the term of linearization error $r$ in \eqref{eq:linearized_with_error_apped2} - via 
\begin{equation}
\label{eq:approximated_dynamics}
    x_{i+1 \mid k}  \approx f_d(z_{i \mid k},v_{i \mid k}) + A_{i \mid k}(x_{i \mid k}-z_{i \mid k})
    +  B_{i \mid k}(u_{i \mid k}-v_{i \mid k})  + W_{i \mid k}.
\end{equation}
Based on the approximated linearized dynamics \eqref{eq:approximated_dynamics}, we find the following LTV error dynamics
\begin{equation}
\begin{aligned}
\Delta x_{i+1 \mid k} &\approx e_{i+1 \mid k}.\\
\end{aligned}
\end{equation}
Hence the predicted state can be formulated as 
\begin{equation}
\label{eq:predicted_state}
    x_{i+1 \mid k} \approx z_{i+1 \mid k}+A_{cl, i} e_{i \mid k}+ W_{i \mid k}.
\end{equation}
Each of the terms evolves under the feedback control \eqref{eq:control_law}
$$
\begin{array}{ll}
z_{i+1 \mid k}=f_d(z_{i \mid k}, v_{i \mid k}) & z_{0 \mid k}=x_{k} \\
e_{i+1 \mid k}=A_{cl, i } e_{i \mid k}+ W_{i \mid k} & e_{0 \mid k} = 0 \\
v_{i \mid k}=K_{i \mid k} z_{i \mid k}+c_{i\mid k}, & 
\end{array}
$$
where $v_{i \mid k}$ is the predicted nominal input. As systems' behavior is predicted within a finite prediction horizon, we let $\mathbf{c}_k=[c_{0 \mid k}^{\top}, \ldots, c_{N-1 \mid k}^{\top}]^{\top}$ and set $c_{i \mid k}=0$ for all $i \geq N$ to ensure a finite number of decision variables.
Similarly, we denote $\mathbf{v}_{k}$ and $\mathbf{z}_{k}$
for the prediction problem with the horizon $N$.

\begin{remark}
    We will also introduce the error propagation with an explicit consideration of the linearization errors in Section \ref{subset:linear_error_prop}. However, we will consider only the dynamic ambiguity propagation in this paper for the interest of practical application.
\end{remark}

Now we define the objective function for the linearized dynamics at time $k$ as $
\mathbb{E}_{\mathbb{P}} \{\sum_{i=0}^{N-1} (\|z_{i\mid k} + e_{i\mid k}\|_{Q}^{2}  + \|c_{i\mid k} + K( z_{i\mid k} + e_{i\mid k}) \|_{R}^{2}) +  \| z_{N\mid k} + e_{N\mid k}\|_{Q_f}^{2}\}$.
Here $Q, Q_f \in \mathbb{R}^{n_x \times n_x}$ and $R \in \mathbb{R}^{n_u \times n_u} $ are positive definite penalty matrices for the quadratic stage costs. Furthermore, based on the assumption of zero-mean additive disturbances, all the accumulated errors $e_{i \mid k} \forall i \in \mathbb{N}_{0}^{N}$ are also zero mean. We could hence reformulate the objective function equivalently to $\sum_{i=1}^{N-1} \| z_{i\mid k}\|_{Q}^{2}+  \| v_{i\mid k}\|_{R}^{2} +  \| z_{N \mid k}\|_{Q_f}^{2} $.

Next, we consider the closed-loop propagation of additive disturbances under closed-loop matrices $A_{cl, i}, \forall \mathbb{N}_{1}^{N}$. For expected constraints satisfaction \eqref{eq:proto_constraint}, we roll out the predicted state ($i \ge 1$) in terms of additive disturbances within the prediction horizon as
\begin{equation}
    x_{i \mid k} \approx z_{i \mid k} + e_{i \mid k} = z_{i \mid k} + \sum_{m = 0}^{i-1}\prod_{j=0}^{i-1-m} A_{cl,i-1-j}^{\min\{1,i-1-m-j\}}W_{m\mid k},
\end{equation}
where $\prod_{j=0}^{i-1-m} A_{cl,i-1-j}^{\min\{1,i-1-m-j\}} = A_{cl,i-1}\dots A_{cl,i-1-m}I$ for $m < i-1$. 
Take $x_{3 \mid k}$ as an example, it can be formulated in terms of additive disturbances via
$x_{3 \mid k } = A_{cl,2}A_{cl,1}A_{cl,0} x_{0 \mid k} + A_{cl,2} A_{cl,1} w_{0 \mid k}  + A_{cl,2}w_{1 \mid k} +  w_{2 \mid k}$.

\begin{remark}
    The relation
\begin{equation}
    \label{eq-error-decomposition}
e_{i \mid k} = \sum_{m = 0}^{i-1}\prod_{j=0}^{i-1-m} A_{cl,i-1-j}^{\min\{1,i-1-m-j\}}W_{m\mid k}
\end{equation}
is the key to our ambiguity set propagation through the (nonlinear) dynamics.
The most significant difference between our work and existing DRMPC works such as \cite{mark2020stochastic} is that we do not assume having data samples of the predicted states $\hat{x}_{i \mid k}$, which would simply reduce the optimal control problem to static Wasserstein DRO. 
However, in practice, one is often faced with the question of having to predict future state distributions and the corresponding ambiguity.
We shall demonstrate that, in such dynamic settings, the real power of feedback control is to control the size of the dynamic Wasserstein ambiguity sets in a closed-loop fashion, see Fig. \ref{fig:iterative_LQR_error_prop}.
The only previous work considering the setting equivalent to our dynamic ambiguity set propagation is \cite{yang2020wasserstein, aolaritei2023capture},
but only in the much simpler setting of linear systems. Also, in order to solve dynamic ambiguity set propagation, \cite{yang2020wasserstein} solves a relaxed problem called the Wasserstein penalty problem ( without state constraints), which still requires solving a semi-infinite problem.
The technical difficulty that prevents previous works to go beyond that simple setting lies in the very core of Wasserstein DRO reformulation techniques -- it does not treat complex nonlinear objectives as in nonlinear OCPs.
In contrast, this work proposes the first dynamic Wasserstein closed-loop DRC with nonlinear dynamics and constraints.  The Wasserstein distributional reachable set under dynamic propagation will be analytically characterized in Proposition \ref{prop-error-reachable-set}.
\end{remark}

Hence, based on the linearized dynamics, we consider an approximated optimal control problem 
corresponding to the prototype DRNMPC problem \eqref{eq:iLQR_prototype_prob}.
\begin{problem}
\label{eq:linearized_prob}
\begin{equation}
\label{eq:iLQR_linearized_prob}
\begin{array}{cl}
\displaystyle\min_{\mathbf{z}, \mathbf{v},\mathbf{K}} & \sum_{i=0}^{N-1}(\left\|z_{i \mid k}\right\|_{Q}^{2}+\left\|v_{i \mid k}\right\|_{R}^{2}) +\left\| z_{N \mid k}\right\|_{Q_f}^{2} \\
\text { s.t. } 
& z_{0\mid k}=x_{k}, \quad  z_{i+1 \mid k}=f_d(z_{i \mid k}, v_{i \mid k})\\
& v_{i \mid k}=K_{i \mid k} z_{i \mid k}+c_{i\mid k}\\
&    \begin{aligned}
     &\sup_{\mathbb{P}_{m} \in \mathcal{P}_{k + m}, m = 0, \dots, i-1} \mathbb{E}_{\mathbb{P}^{\otimes i}}\{ [F]_{n}( z_{i \mid k} + \sum_{m = 0}^{i-1}A_{cl}^{(i,j,m)}W_{m\mid k}) \} \leq [f]_{n}
     \end{aligned}\\
& \forall \quad i \in \mathbb{N}_1^N, n \in \mathbb{N}_1^{n_F},  k \in \mathbb{N}_{\ge 0},\\ 
\end{array}
\end{equation}
\end{problem}
where $W_{m \mid k}\sim\mathbb{P}_{m}$ is the disturbance variable and $A_{cl}^{(i,j,m)} := \prod_{j=0}^{i-1-m}A_{cl,i-1-j}^{\min\{1,i-1-m-j\}}$.

In the following, we will provide the exact reformulation of the optimization problem \eqref{eq:iLQR_linearized_prob}. Before showing the final reformulation, we require the following Lemma to reformulate the distributionally robust constraints.

\begin{lemma}
\label{lemma:1}
Consider the polytopic uncertainty set $\mathbb{W}_{w}$ and the Wasserstein ambiguity set $\mathcal{P}$ as the Wasserstein ball around the empirical distribution $\hat{\mathbb{P}}=\frac{1}{M} \sum_{l=1}^{M} \delta_{\hat{w}^{(l)}}$ with type-1 Wasserstein metric and ball radius $\varepsilon$. Then, the worst-case expectation 
$
\sup_{\mathbb{P}_{m} \in \mathcal{P}} \mathbb{E}_{\mathbb{P}_{m}}\left\{[F]_{n}( A_{cl}^{(i,j,m)}W_{m\mid k})\right\}
$
evaluates to 
\begin{equation}
\label{eq:DRO_constraint_tightening}
    \begin{aligned}
      \inf_{\lambda, s_{l}, \gamma_{l}} & \lambda \varepsilon + \frac{1}{N} \sum_{l=1}^{M}s_{l} \\
     & \begin{aligned}
     \text{s.t.} & \quad [F]_{n}(A_{cl}^{(i,j,m)})\hat{w}_{l} + \gamma_{l}^{\top}(h-H\hat{w}_{lm}) \le s_{l}\\
     &  \quad\|H^{\top}_{w}\gamma_{l} - \left[[F]_nA_{cl}^{(i,j,m)}\right]^{\top}\|_{*} \le \lambda\\
     & \quad\gamma_{l} \ge 0, \quad \forall  l \in \mathbb{N}_{1}^{M},
     \end{aligned}
     \end{aligned}
\end{equation}
where $\lambda \in \mathbb{R}, s_{l} \in \mathbb{R}, \gamma_{l} \in \mathbb{R}^{n_H}$, and $\|\cdot \|_{*}$ is the dual norm corresponding to the norm applied in \eqref{eq:wasserstein_metric}.
\end{lemma}
\begin{proof}
The equivalent reformulation can be derived with $a_k := [F]_{n}A_{cl}^{(i,j,m)}$ in \cite[Corollary 5.1]{mohajerin2018data}.  
\end{proof}

\begin{proposition}
\label{prop:1}
Consider the polytopic uncertainty set $\mathbb{W}_{w}$. Then, the DRMPC problem \eqref{eq:iLQR_linearized_prob} evaluate to
\begin{equation}
\label{eq:optimal_control_tightening}
\begin{array}{cl}
\displaystyle\min_{\mathbf{z}, \mathbf{v},\mathbf{K}, {\lambda_{m}, s_{ml}, \gamma_{ml}}} & \sum_{i=0}^{N-1}(\left\|z_{i \mid k}\right\|_{Q}^{2}+\left\|v_{i \mid k}\right\|_{R}^{2}) +\left\| z_{N \mid k}\right\|_{Q_f}^{2} \\
\text { s.t. } 
& z_{0\mid k}=x_{k} \\
&  z_{i+1 \mid k}=f_d(z_{i \mid k}, v_{i \mid k}) \\
& v_{i \mid k}=K_{i\mid k} z_{i \mid k}+c_{i\mid k}\\
&   \sum_{m = 0}^{i-1}  \lambda \varepsilon + \frac{1}{N} \sum_{l=1}^{M}s_{ml} 
\leq [f]_{n} -  [F]_{n}z_{i \mid k}\\
     &   [F]_{n}A_{cl}^{(i,j,m)}\hat{w}_{l} + \gamma_{l}^{\top}(h-H\hat{w}_{l}) \le s_{ml}\\
     &  \|H^{\top}_{w}\gamma_{l} - [[F]_n A_{cl}^{(i,j,m)} ]^{\top}\|_{*} \le \lambda_{m}\\
     & \gamma_{ml} \ge 0, \quad \forall l \in \mathbb{N}_{1}^{M}, \forall m \in \mathbb{N}_{0}^{i-1}\\
& \forall \quad i \in \mathbb{N}_1^N, n \in \mathbb{N}_1^{n_F},  k \in \mathbb{N}_{\ge 0}.\\ 
\end{array}
\end{equation}
\end{proposition}

\begin{proof}
    The reformulation  \eqref{eq:optimal_control_tightening} is the consequence of the exact reformulation of the distributionally robust constraints. For any given $i$ and $k$, based on the linearity property of expectation,  the distributionally robust constraints $\sup_{r_{m\mid k} , \mathbb{P}_0 \in \mathcal{P},\dots,\mathbb{P}_{i-1} \in \mathcal{P}} \mathbb{E}_{\mathbb{P}^{\otimes i}}\{ [F]_{n}( z_{i \mid k} + \sum_{m = 0}^{i-1}A_{cl}^{(i,j,m)} W_{m\mid k}) \} \leq [f]_{n}$
    is  equivalent to 
    \begin{equation}
    \label{eq:reform_DR_constraints}
\begin{aligned}
 & \sum_{m = 0}^{i-1}\sup_{\mathbb{P}_m \in \mathcal{P}}  \mathbb{E}_{\mathbb{P}_{m}} \left\{[F]_{n}( A_{cl}^{(i,j,m)}W_{m\mid k})\right\} + [F]_{n}z_{i \mid k} \le [f]_n. 
 \end{aligned}
 \end{equation}


Let $\mathbb{Z}_{f} := \{\mathbf{z} \mid \mathbf{z} \text{ is feasible in } \eqref{eq:reform_DR_constraints}\ \forall i \in \mathbb{N}_{1}^{N}\}$.
By applying Lemma \ref{lemma:1} to $
\sup_{\mathbb{P}_{m} \in \mathcal{P}} \mathbb{E}_{\mathbb{P}_{m}}\left\{[F]_{n}( A_{cl}^{(i,j,m)} W_{m\mid k})\right\}
$ for each $m$, we acquire that the inequality \eqref{eq:reform_DR_constraints} containing the summation of distributionally robust optimizations is equivalent to
 \begin{equation}
 \label{eq:multi_DRO_after_reform}
 \begin{aligned}
 \displaystyle  \sum_{m = 0}^{i-1}  & \inf_{\lambda_{m}, s_{ml}, \gamma_{ml}} \lambda \varepsilon + \frac{1}{N} \sum_{l=1}^{M}s_{ml} \le [f]_{n} - [F]_{n} z_{i \mid k}\\
   & \quad \quad\begin{aligned}
     \text{s.t.}\quad &   [F]_{n}A_{cl}^{(i,j,m)}\hat{w}_{l} + \gamma_{l}^{\top}(h-H\hat{w}_{l}) \le s_{ml}\\
     &  \|H^{\top}_{w}\gamma_{l} - [[F]_n A_{cl}^{(i,j,m)} ]^{\top}\|_{*} \le \lambda_{m}\\
     & \gamma_{ml} \ge 0, \quad \forall l \in \mathbb{N}_{1}^{M}, \forall m \in \mathbb{N}_{0}^{i-1}.\\
     \end{aligned}\\
    \end{aligned} 
\end{equation}
Hence the feasible set $\mathbb{Z}_f$ is equivalent to 
 $$
\begin{aligned}
 \mathbb{Z}_{f} := \{&\mathbf{z} \mid \exists \lambda_m, s_{ml}, \gamma_{ml} \text{ s.t. } \mathbf{z} \text{ is feasible in } \\
 & \quad \begin{aligned}
 &\displaystyle  \sum_{m = 0}^{i-1}    \lambda \varepsilon + \frac{1}{N} \sum_{l=1}^{M}s_{ml} \le [f]_{n} - [F]_{n} z_{i \mid k}\\
     &   [F]_{n} A_{cl}^{(i,j,m)}\hat{w}_{l} + \gamma_{l}^{\top}(h-H\hat{w}_{l}) \le s_{ml}\\
     &  \|H^{\top}_{w}\gamma_{l} - [[F]_n A_{cl}^{(i,j,m)} ]^{\top}\|_{*} \le \lambda_{m}\\
     & \gamma_{ml} \ge 0, \quad \forall l \in \mathbb{N}_{1}^{M}, \forall m \in \mathbb{N}_{0}^{i-1}\\
    \end{aligned} \\
 & \forall i \in \mathbb{N}_{1}^{N}\}.
 \end{aligned}\\
 $$
 Together with the objective function and equality constraints corresponding to the nominal dynamic, we complete the proof.
\end{proof}

\begin{remark}
    All the results about the worst-case expected constraints in this paper can be easily replaced by the worst-case chance constraints via the CVaR formulation introduced by \cite{hota2019data, xie2021distributionally}.
    Also, this work can be easily extended to the control problem with polytopic input constraints, e.g. \cite{zhong2023tube}.
\end{remark}

\subsection{Iterative distributionally robust LQR}
Optimization problem \eqref{eq:iLQR_linearized_prob}
is difficult to solve as the matrices of the linearized system depend on the unknown nominal trajectory and the back-off $\beta$ is dependent on the unknown system matrices. Hence, we propose in the following an algorithm iterating by sequentially 
deriving the linearized system and the corresponding feedback gains via Riccati recursion \cite[Sec. 8.8.3]{rawlings2017model}, solving a perturbed nominal problem, and updating the perturbations based on the solution from a DRO problem. 

We consider below the setting of distributionally robust optimal control, i.e. the problem \eqref{eq:optimal_control_tightening} with the initial sampling time $k = 0$ and a fixed 
prediction horizon $N$. Given initial trajectories of nominal state and input $\bar{\boldsymbol{z}}$, $\bar{\boldsymbol{v}}$ (e.g. nominal nonlinear MPC), we solve an iterative LQR problem by Riccati Recursion \cite{li2004iterative, mayne1966second, rawlings2017model} to get the matrices $\mathbf{A},\mathbf{B},\mathbf{K}$ corresponding to the linearized system, where $\mathbf{A} = \{A_{0 \mid k},\dots, A_{N \mid k}\}$, $\mathbf{B} = \{B_{0 \mid k},\dots, B_{N \mid k}\}$, and $\mathbf{K} = \{K_{0 \mid k},\dots, K_{N \mid k}\}$. The feedback gain derived from Riccati recursion makes the closed-loop system matrix $A_{cl}$ stable locally around the linearization point.
Then, we solve \eqref{eq:DRO_constraint_tightening} to update the back-off $\mathbf{\beta} =  \{\beta_{0 \mid k},\dots, \beta_{N \mid k}\}$, where
    \begin{equation}
    \label{eq:beta}
\begin{aligned}
 \beta_{i \mid k} := \sup_{\mathbb{P}_{m} \in \mathcal{P}, m = 0, \dots, i-1} &\mathbb{E}_{\mathbb{P}^{\otimes i}}\{[F]_{n}(\sum_{m = 0}^{i-1}A_{cl}^{(i,j,m)} W_{m\mid k})\}.
 \end{aligned}
 \end{equation}
 After determining the back-off, we solve the constraint-tightened program
\begin{equation}
\label{eq:optimal_control_tightening_fixed_matirces}
\begin{array}{cl}
\displaystyle\min_{\mathbf{z}, \mathbf{v}} & \sum_{i=0}^{N-1}(\left\|z_{i \mid k}\right\|_{Q}^{2}+\left\|v_{i \mid k}\right\|_{R}^{2}) +\left\| z_{N \mid k}\right\|_{Q_f}^{2} \\
\text { s.t. } 
& z_{0\mid k}=x_{k}\\
& z_{i+1 \mid k}=f_d(z_{i \mid k}, v_{i \mid k}) \\
& v_{i \mid k}=K_{i\mid k} z_{i \mid k}+c_{i\mid k}\\
&[F]_{n}z_{i \mid k} \leq [f]_{n} - \beta_{i \mid k}, \\
& \forall \quad i \in \mathbb{N}_1^N, n \in \mathbb{N}_1^{n_F},  k = 0.\\ 
\end{array}
\end{equation}
We iterate the process above until convergence. The proposed algorithm is summarized in Algorithm \ref{alg:1}.

\begin{algorithm}[H]
 \caption{Iterative distributionally robust  LQR} \label{alg:1}
 \begin{algorithmic}[1]
   \State \textbf{Input}: Initial guess $\bar{\boldsymbol{z}}, \bar{\boldsymbol{v}}, \varepsilon, \{\hat{w}\}_{i=1}^{M}$
    \While{Not Converge}
    \State $\boldsymbol{A}, \boldsymbol{B}, \boldsymbol{K}
    \leftarrow$ Riccati Recursion $(\bar{\boldsymbol{z}}, \bar{\boldsymbol{v}})$ \cite[Sec. 8.8.3]{rawlings2017model}
    \State $\boldsymbol{\beta}$ $\leftarrow$ DRO ($\boldsymbol{A}, \boldsymbol{B}, \boldsymbol{K}, \varepsilon, \{\hat{w}\}_{i=1}^{M}$) \eqref{eq:DRO_constraint_tightening} 
  \State $\bar{\boldsymbol{z}}, \bar{\boldsymbol{v}},\bar{\boldsymbol{c}}$ $\leftarrow$ Nonlinear optimal control problem with fixed $\boldsymbol{A}, \boldsymbol{B}, \boldsymbol{K}, \boldsymbol{\beta}$ \eqref{eq:optimal_control_tightening_fixed_matirces} 
    \EndWhile
    \State Return: $\bar{\boldsymbol{z}}, \bar{\boldsymbol{v}}, \bar{\boldsymbol{c}}, \boldsymbol{K}$
 \end{algorithmic} 
 \end{algorithm}

\begin{remark}
For the original distributionally robust MPC problem \eqref{eq:iLQR_prototype_prob}, we could recursively solve Algorithm \ref{alg:1} at each sampling time $k$ for the given measurement $x_k$. The initial guess could be generated from a nominal NMPC or a shifted trajectory derived from the previous step.
\end{remark}

\subsection{Wasserstein distributional reachable sets for error dynamics}
\label{sec:err-reachable}
We now further show that the stochastic error characterized by the LTV dynamics 
\begin{equation}
   \small
    \label{eq-error-dynamics-ltv}
    e_{i+1 }=A_{cl, i} e_{i }+ W_{i}, \  W_{i}\overset{\textrm{i.i.d.}}{\sim} \mathbb{P}_w
\end{equation}
can be contained in a dynamic Wasserstein ambiguity set given below.
This gives a theoretical bound for the experimental results in Figure~\ref{fig:iterative_LQR_error_prop}.

Let $\hat e_i\sim \hat Q_{N,i}$ be the empirical error vector at predicted time step $i$,
and
$ \hat Q_{N,i}$ be its empirical error distribution, given by the empirical LTV error dynamics
$$
    \small
    \hat e_{i+1 }=A_{cl, i} \hat e_{i }+ \hat W_{i}, \ 
    e_{0 } = 0,
    \hat W_{i}\overset{\textrm{i.i.d.}}{\sim} \hat{\mathbb{P}}
    .
$$
We now consider a dynamic ambiguity set\textemdash\emph{Wasserstein ambiguity tube} associated with the LTV dynamics,
with a slight abuse of notation,
$$
\small
\begin{aligned}
    & \mathbb{T}_\epsilon(\hat P_N)
    :=
        \Bigl\{
            Q_i, i = 0,\dots, T
            \ 
            \vert
            \
            e_{i }\sim Q_i,\
            e_{i+1 }=A_{cl, i} e_{i }+ W_{i},
             W_{i}\sim \mathbb{P}_{w},
            d_w^{p}\left(\mathbb{P}_w, \hat {\mathbb{P}}\right)\leq \epsilon,
            e_{0 } = 0
        \Bigr\}
    .
    \end{aligned}
    $$

The intuition is that the Wasserstein ambiguity tube contains all evolution paths of the ambiguous stochastic system, i.e., the state distribution of our MPC problem lives in this ambiguity tube
$\{Q_i\}\in\mathbb{T}_\epsilon(\hat {\mathbb{P}})$.

The following result characterizes the size of the Wasserstein ambiguity tube.
\begin{proposition}
    [Wasserstein distributional reachable sets]
    \label{prop-error-reachable-set}
    We have,
    $\forall \{Q_i\}\in\mathbb{T}_\epsilon(\hat {\mathbb{P}})$,
    $$
        \small
        \begin{aligned}
        & d_w^{p}\left({Q_i}, \hat Q_{N,i}\right)
             \leq 
            \epsilon \cdot  
            \sum_{m = 0}^{i-1}\prod_{j=0}^{i-1-m} 
             \|A_{cl,i-1-j}^{\min\{1,i-1-m-j\}}\|^p, \\
         \end{aligned}
         $$
    $\textrm{ for } i=0,\dots, T$, where $\|\cdot \|$ is the corresponding induced matrix norm.
\end{proposition}
\begin{proof}
    By the definition of the Wasserstein distance,
    $$
    \small
        d^{p}_{w}\left({Q_i}, \hat Q_{N,i}\right) =\inf_{\Pi} \int_{\mathbb{W}_{w}^{2}}
        \left\|
            e^*_i-\hat e_i
        \right\|^{p} 
        \mathrm{d} \Pi\left( e^*_i, \hat e_i\right)
        ,
    $$
    where $\Pi$ is a joint distribution (transport plan) with marginals ${Q_i}, \hat Q_{N,i}$.

    Plugging in the error dynamics decomposition~\eqref{eq-error-decomposition},
    $$
    \small
    \begin{aligned}
        &\inf_\Gamma
        \int_{\mathbb{W}_{w}^{2}}
            \left\|
                \sum_{m = 0}^{i-1}\prod_{j=0}^{i-1-m} A_{cl,i-1-j}^{\min\{1,i-1-m-j\}}(W^*_{m} - \hat W_{m})
            \right\|^{p}  \mathrm{d} {\Gamma}\left( W^*_{m}, \hat W_{m}\right)
            \\
           & \leq
            \sum_{m = 0}^{i-1}
            \left\|
            \prod_{j=0}^{i-1-m} 
                A_{cl,i-1-j}^{\min\{1,i-1-m-j\}}
                \right\|
                ^{p} 
            \cdot \inf_\Gamma
            \int_{\mathbb{W}_{w}^{2}}
            \left\|
                W^*_{m} - \hat W_{m}
            \right\|^{p} 
            \mathrm{d} \Gamma\left( W^*_{m}, \hat W_{m}\right)
            \\
          &  \leq 
            \epsilon
            \cdot 
            \sum_{m = 0}^{i-1}\prod_{j=0}^{i-1-m} 
            \left\|
                A_{cl,i-1-j}^{\min\{1,i-1-m-j\}}
            \right\|^{p} 
            ,
            \end{aligned}
            $$
        where $\Gamma$ is the joint distribution of $W^*_{m}, \hat W_{m}$.
        Due to the dynamics structure~\eqref{eq-error-dynamics-ltv}, the joint distribution $\Pi$ is determined by the joint distributions $\Gamma$.
        The last inequality estimate above is due to the Wasserstein distance estimate
        $d^{p}_{w}\left(\mathbb{P}_w, \hat {\mathbb{P}}\right)\leq \epsilon$.
\end{proof}
Proposition~\ref{prop-error-reachable-set} equips us with a reachable set for the error in the Wasserstein distance. We illustrate this in a Figure~\ref{fig:iterative_LQR_error_prop}.
It further implies that the aforementioned ambiguity tube can be bounded in a more straightforward and computable dynamic ambiguity set (tube)
\begin{equation}
\small
\begin{aligned}
    \mathbb{T}_\epsilon(\hat {\mathbb{P}})
    \subset
    \Bigl\{&
        Q_i, i = 0,\dots, T
        \ 
        \vert
        \
    d^{p}_{w}\left({Q_i}, \hat Q_{N,i}\right)
        \leq 
        \epsilon \cdot  
        \sum_{m = 0}^{i-1}\prod_{j=0}^{i-1-m} 
         \|A_{cl,i-1-j}^{\min\{1,i-1-m-j\}}\|^p,
         \\ 
     &    \hat e_{i }\sim \hat Q_{N,i},\ 
         \hat e_{i+1 }=A_{cl,i} \hat e_{i }+ \hat W_{i}, \ 
         \hat W_{i}\overset{\textrm{i.i.d.}}{\sim} \hat {\mathbb{P}},
         e_{0 } = 0.
    \Bigr\}
    \
    \end{aligned}
\end{equation}    
Unlike aforementioned works in the existing literature where the ambiguity sets are often given a priori, our dynamic Wasserstein ambiguity set is obtained by propagating through the LTV error dynamics.

\begin{remark}
    [Wasserstein invariant ambiguity sets]
    While this paper does not deal with infinite-horizon control or positive invariant sets, it is easy to see that Proposition~\ref{prop-error-reachable-set} can be used to construct Wasserstein invariant ambiguity sets by examining the series 
    $\sum_{m = 0}^{\infty}\prod_{j=0}^{m} \|A_{cl,i-1-j}^{\min\{1,i-1-m-j\}}\|^p$, i.e.,
    if $\exists C <\infty$ such that $\sum_{m = 0}^{\infty}\prod_{j=0}^{m} \|A_{cl,i-1-j}^{\min\{1,i-1-m-j\}}\|^p \leq C$, then the following set of distributions is a Wasserstein invariant set for the ambiguous system state distribution
    $
    \Bigl\{
            Q
            \
            \vert
            \
            d^{p}_{w}\left({Q}, \hat Q_{N,\infty}\right)
            \leq 
            \epsilon \cdot C
        \Bigr\}
    $,
    where $\hat Q_{N,\infty}$ is the equilibrium state distribution of the nominal error dynamics~\eqref{eq-error-dynamics-ltv}.
\end{remark}

\subsection{Linearization error reachable sets}
\label{subset:linear_error_prop}
Let us now consider the error dynamics propagation for linearization errors with the following standard assumption.
\begin{assumption}
    \label{assump:3_times_diff_append2}
    The nonlinear dynamics \eqref{eq:system} $f_d: \mathbb{R}^{n_{\mathrm{x}}} \times$ $\mathbb{R}^{n_{\mathrm{u}}} \mapsto \mathbb{R}^{n_{\mathrm{x}}}$ are three times continuously differentiable.
\end{assumption}
To bound the linearization error, we consider the following condition of locally bounded eigenvalues on Hessian, similar as in \cite{leeman2023robust}. Let $\mathcal{X} \in \mathbb{R}^{n_x}$ and $\mathcal{U} \in \mathbb{R}^{n_u}$ denote the state and input space, respectively and  $H_n: \mathbb{R}^{n_{\mathrm{x}}+n_{\mathrm{u}}} \mapsto \mathbb{R}^{\left(n_{\mathrm{x}}+n_{\mathrm{u}}\right) \times\left(n_{\mathrm{x}}+n_{\mathrm{u}}\right)}$ denote the Hessian corresponding to the n-th component of $f_d$, i.e.
$$
H_n(\xi_x, \xi_u)=\left.\left[\begin{array}{cc}\frac{\partial^2 f_{d,n}}{\partial x^2} & \frac{\partial^2 f_{d,n}}{\partial x \partial u} \\ * & \frac{\partial^2 f_{d,n}}{\partial u^2}\end{array}\right]\right|_{(x, u)=(\xi_x, \xi_u)},
$$
where $\xi_x \in \mathcal{X}$ and $\xi_u \in \mathcal{U}$. We further denote the constant $\mu_{n}$ as the corresponding locally maximal eigenvalue, i.e. 
$$
\mu_n:=\frac{1}{2} \max _{\xi_x \in \mathcal{X},\xi_u \in \mathcal{U},\|h\|_{\infty} \leq 1}\left|h^{\top} H_n(\xi_x, \xi_u) h\right|. 
$$

Then we have the following bound for each n-th element of the vector of linearization errors $r(x,u,z,v)$ in \eqref{eq:linearized_with_error_apped2}.
\begin{lemma}\cite[Proposition III.1.]{leeman2023robust}
\label{lemma:ri_bound_append2}
Given Assumption \ref{assump:3_times_diff_append2}, the remainder in \eqref{eq:linearized_with_error_apped2} satisfies
$$
\left|r_n(x, u, z, v)\right| \leq\|\eta\|_{\infty}^2 \mu_n,
$$
for any $z, x \in \mathcal{X}, u, v \in \mathcal{U}$, where $\eta = \begin{bmatrix}
    \Delta x\\
    \Delta u
\end{bmatrix}$.
\end{lemma}
\begin{proof}
    By second-order Mean Value Theorem, we know that there exist $\xi_x \in [x,z]$ and $\xi_u \in [u,v]$ (with a little abuse of notation) such that $r_n(x, u, z, v) = 
\frac{1}{2}\begin{bmatrix}
    x-z \\ u-v 
\end{bmatrix}^T H_n(\xi_x, \xi_u)\begin{bmatrix}
    x-z \\ u-v 
\end{bmatrix}$. Hence we have 
$$
\begin{aligned}
\left|r_n(x, u, z, v)\right| &\le \max_{\xi_x, \xi_u} \frac{1}{2} \left| \begin{bmatrix}
    \Delta x \\  \Delta u 
\end{bmatrix}^T H_n(\xi_x, \xi_u)\begin{bmatrix}
    \Delta x \\ \Delta u 
\end{bmatrix} \right|\\
& \le \left\| \begin{bmatrix}
    \Delta x \\  \Delta u 
\end{bmatrix} \right\|^2 \max_{\xi_x \in \mathcal{X}, \xi_u \in \mathcal{U}, \|h \| \le 1} \frac{1}{2} \left| h^T H_n(\xi_x, \xi_u) h\right|.\\
\end{aligned}
$$
If the infinity norm is considered here, we have $
\left|r_n(x, u, z, v)\right| \leq\|\eta\|_{\infty}^2 \mu_n
$.
\end{proof}

Given Lemma \ref{lemma:ri_bound_append2},  dynamics \eqref{eq:error_dynamic_append2} and the control policy, we have the following lemma characterizing the upper bound of the linearization error.

\begin{lemma}
    \label{lemma: error_dynamic_upper_bound}
    Given Assumption \ref{assump:3_times_diff_append2}, the remainder in \eqref{eq:linearized_with_error_apped2} is upper bounded by
    \begin{equation} 
\begin{aligned}
    r_{n}(x_{k+i},u_{k+i},z_{k+i},v_{k+i}) \le \left(\left\|\begin{bmatrix}
        e_{k+i} \\ K_{k+i} e_{k+i}
    \end{bmatrix} \right\|_{\infty} + \left\|
    \begin{bmatrix}
    \varepsilon_{k+i} \\
    K_{k+i}\varepsilon_{k+i} \end{bmatrix}\right\|_{\infty}\right)^{2}\mu_n, 
    \end{aligned}
    \end{equation}
for any $z, x \in \mathcal{X}, u, v \in \mathcal{U}$, where $i \in  \mathbb{N}_{+}$ and $n \in \mathbb{N}_{1}^{n_x}$.
\end{lemma}
\begin{proof}
    This can be directly derived from Lemma \ref{lemma:ri_bound_append2} with the error dynamics \eqref{eq:error_dynamic_append2}.
\end{proof}

Furthermore, let $
\mu:=\operatorname{diag}\left(\mu_1, \ldots, \mu_{n_{\mathrm{x}}}\right)$, then the linearization error satisfies
$$
\begin{aligned}
r_{k + i} \in \left(\left\|\begin{bmatrix}
        e_{k+i} \\ K_{k+i} e_{k+i}
    \end{bmatrix} \right\|_{\infty} + \left\|
    \begin{bmatrix}
    \varepsilon_{k+i} \\
    K_{k+i}\varepsilon_{k+i} \end{bmatrix}\right\|_{\infty}\right)^{2} \mu  \mathbb{B}_{\infty}^{n_{\mathrm{x}}},
\end{aligned}
$$
where $r_{k+i} := \begin{bmatrix}
    r_{1}(x_{k+i}, u_{k+i}, z_{k+i}, v_{k+i})\\
    \vdots\\
    r_{n_x}(x_{k+i}, u_{k+i}, z_{k+i}, v_{k+i})
\end{bmatrix}$. 


\begin{remark}
    The back-off due to the linearization error can be derived from Lemma \ref{lemma: error_dynamic_upper_bound}. 
    However, we observed that, in practical experiments, the linearization error propagation might result in an over-conservative closed-loop performance; hence, we will only consider the dynamic propagation of ambiguity sets in practical numerical experiments below.
\end{remark}

\section{Case study}
The system considered is a nonlinear mass spring system  with $m = 2\,\text{kg}, k_1 = 3\, \text{N/m}, k_2 = 2 \,\text{N/m}$:
$$
\begin{aligned}
\dot{x}_{1} &= x_{2}\\
\dot{x}_{2} &= -\frac{k_{2}}{m} x_{1}^{5} -\frac{k_{1}}{m} x_{2}+\frac{1}{m} u.
\end{aligned}
$$
The discrete-time system is acquired by using the Runge-Kutta method with fourth order with the sampling period $0.1\,\text{s}$. We simulate the control performance for the discrete-time system suffering from the uniformly distributed additive disturbance bounded within $[-1e-3, 1e-3]$ on the state element $x_1$, and $[-0.1, 0.1]$ on $x_2$. The prediction horizon for this system is set to $N = 140$. 

The control goal of this system is to track the state $x_r = [0, 0]^{\top}$ starting from the initial state $x_{init} = [-2,0]^{\top}$, while satisfying the distributionally robust state constraint corresponding to  $x_2 \le 0.5\,\text{m/s}$. The parameters are selected as $Q = Q_f =  \begin{bmatrix}
100 & 0 \\
0 & 1
\end{bmatrix}$, $R = [1]$. We compare the closed-loop performance of three different methods: our method, fixed feedback gain, and no feedback gain in this section. The fixed feedback gain $K=[-7.97, -7.16]$ is derived from the LQR controller for the nonlinear system linearized around the equilibrium point with the same penalty matrices. 

With $M=5$ offline collected disturbance samples and ball radius $\varepsilon = 0.03$, simulation results of nominal trajectories solved by Algorithm \ref{alg:1} with the three different methods mentioned above (for the method with the fix or zero feedback gain we only linearize the nominal nonlinear system without update $K$) can be found in fig. \ref{fig:iterative_LQR}. We apply $u_{k} = K_{k} x_{k} + c_{k}$ to the disturbed nonlinear system from $k=0$ to $k = N$, where $K_k = 0$ for the method with zero feedback gain and $K_k = [-7.97, -7.16]$ for the method with fixed feedback gain. Fig \ref{fig:iterative_LQR_multiple_realization} illustrates 20 realizations of the closed-loop performance. Fig. \ref{fig:iterative_LQR_error_prop} illustrates the accumulated error between the closed-loop state and the nominal state shown in fig. \ref{fig:iterative_LQR}.  The arrow indicates the error difference between two consecutive sampling times, i.e. the tail indicates the accumulated error with respect to the nominal state of the previous steps and the head indicates the accumulated error with respect to the nominal state at the current step.

Based on our experiments, we conclude that though the open-loop controller is capable of controlling the nominal trajectory,  it is not effective in reducing the growth of ambiguity, as depicted in Fig \ref{fig:iterative_LQR_error_prop}. Conversely, our feedback controller successfully controls the size of ambiguity, which is the main insight of our paper.


\begin{figure}[thpb]
  \centering
\includegraphics[width=0.4\textwidth, height = 0.2\textheight]{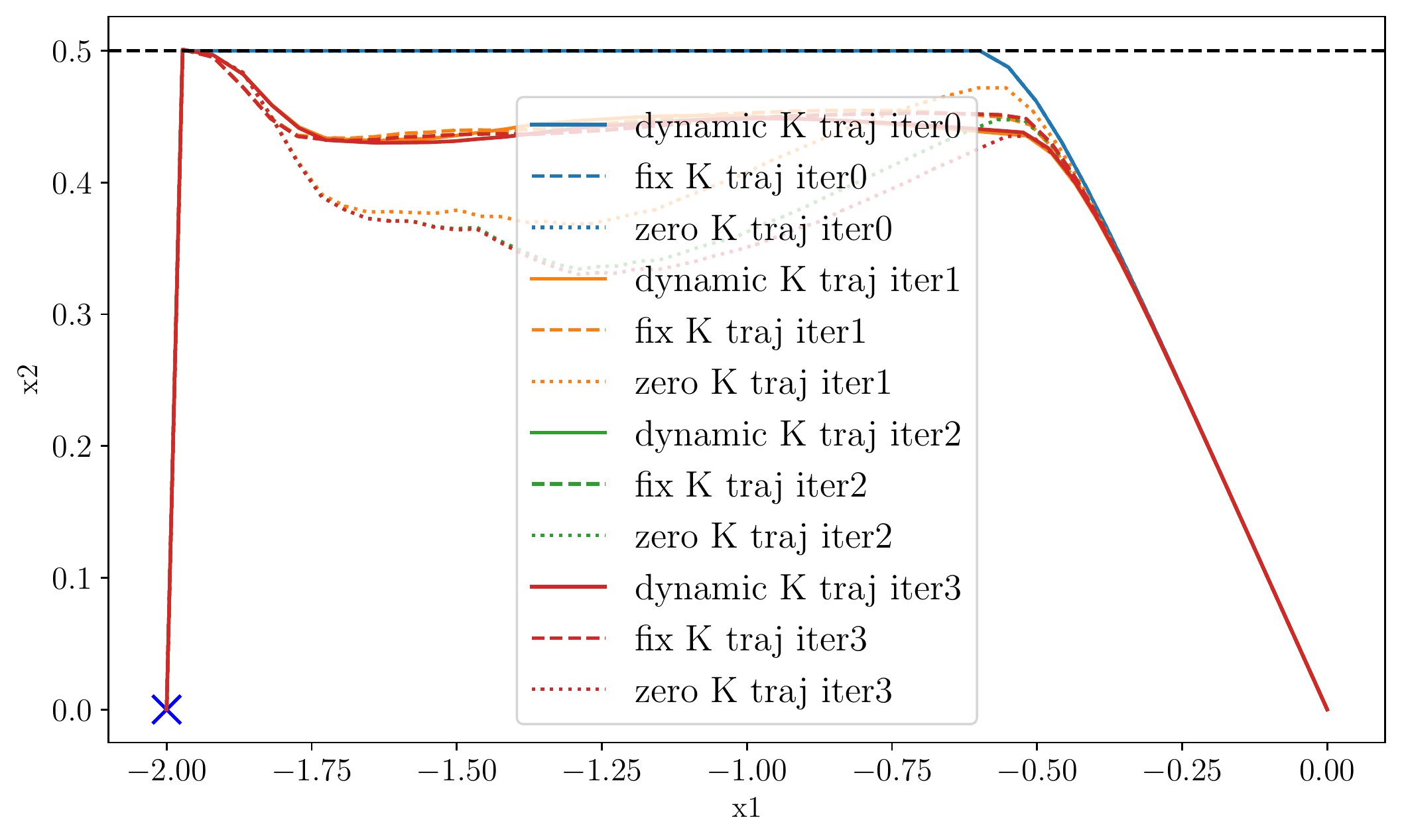}
    \caption{Nominal state trajectories solved by Algorithm \ref{alg:1} with Wasserstein ball radius $\varepsilon = 0.03$.}
    \label{fig:iterative_LQR}
\end{figure}    
\begin{figure}[thpb]
  \centering
\includegraphics[width=0.4\textwidth, height = 0.2\textheight]{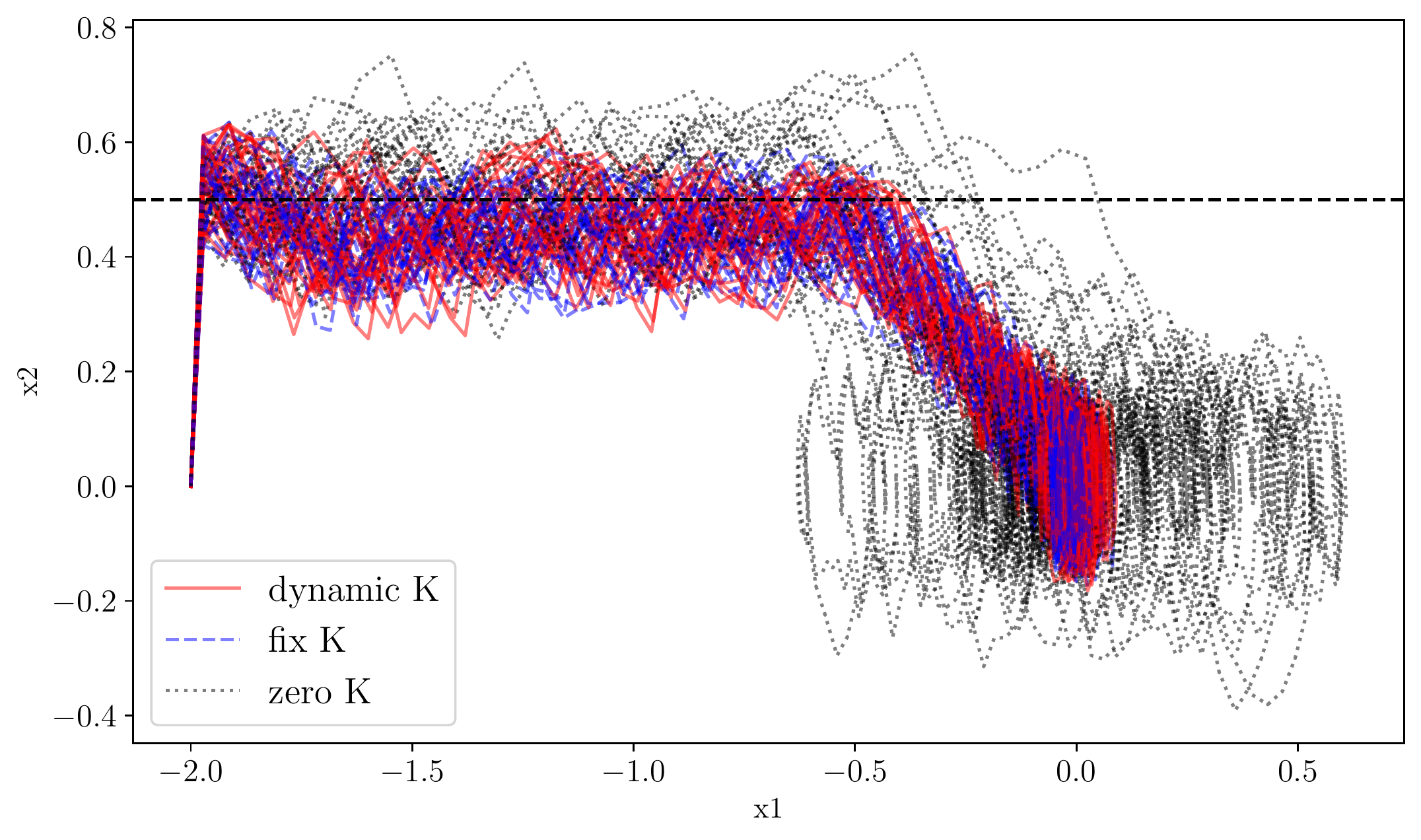}
    \caption{Closed-loop performance of 30 realizations with the feedback gains and nominal inputs solved by \ref{alg:1}. Red: Out method. Blue: Fixed feedback gain. Black: Zero feedback gain.} \label{fig:iterative_LQR_multiple_realization}
\end{figure}

\section{Conclusions}
Our key insight is that the sizes of Wasserstein ambiguity sets for nonlinear systems can be controlled using nonlinear feedback control.
To demonstrate that, this paper proposes the DRNMPC with dynamic Wasserstein ambiguity. We propose an iterative MPC scheme to dynamically control the propagation of Wasserstein ambiguity sets. We analytically characterize the Wasserstein distributional reachable set under dynamic propagation in our algorithm. To evaluate the effectiveness of our proposed algorithm, we compare the closed-loop performances of dynamic feedback, fixed feedback, and no feedback on a mass-spring system. The simulation results demonstrate that the proposed iterative scheme can effectively control the ambiguity set propagation, which is a critical step in solving the DRNMPC problem. 


\bibliographystyle{unsrt}  
\bibliography{root}

\end{document}